\documentclass{llncs}
\usepackage[T1]{fontenc}
\usepackage[utf8]{inputenc}

\usepackage{amsfonts,amsmath,amssymb}
\usepackage{cite,microtype,graphicx,hyperref}
\usepackage[capitalize,nameinlink]{cleveref}

\newcommand{\blowup}[2]{#1#2}

\title{On the Biplanarity of Blowups}

\author{David Eppstein}
\institute{Computer Science Department, University of California, Irvine}

\date{ }

\begin{document}
\thispagestyle{empty}
\maketitle  

\begin{abstract}
The 2-blowup of a graph is obtained by replacing each vertex with two non-adjacent copies; a graph is biplanar if it is the union of two planar graphs. We disprove a conjecture of Gethner that 2-blowups of planar graphs are biplanar: iterated Kleetopes are counterexamples. Additionally, we construct biplanar drawings of 2-blowups of planar graphs whose duals have two-path induced path partitions, and drawings with split thickness two of 2-blowups of 3-chromatic planar graphs, and of graphs that can be decomposed into a Hamiltonian path and a dual Hamiltonian path.

\renewcommand\and{$\cdot$ }
\keywords{Graph thickness \and Split thickness \and Graph blowups \and Kleetopes.}
\end{abstract}

\section{Introduction}

In a 2018 survey on the Earth--Moon problem, Ellen Gethner conjectured that 2-blowups of planar graphs are always biplanar~\cite{Get-GT2-18}. In this paper we refute this conjecture by showing that 2-blowups of iterated Kleetopes are non-biplanar, and more strongly do not have split thickness two.

\subsection{Definitions and preliminaries}
Before detailing our results, let us unpack this terminology: what are Kleetopes, biplanarity and split thickness, and blowups?

\smallskip\paragraph{Polyhedral graphs} are the graphs of convex polyhedra. By Steinitz's theorem, these are exactly the 3-vertex-connected planar graphs~\cite{Ste-EMW-22}. Polyhedral graphs have unique planar embeddings~\cite{Mac-DMJ-37}, whose faces are exactly the \emph{peripheral cycles}, cycles such that every two edges not in the cycle are part of a path with interior vertices disjoint from the cycle~\cite{Tut-PLMS-63}. Every \emph{maximal planar graph} with $\ge 4$ vertices, one to which no edges can be added while preserving planarity, is polyhedral.

The \emph{Kleetope} of a polyhedral graph (named by Branko Grünbaum for Victor Klee~\cite{Gru-IJM-63}) is a maximal planar graph obtained by adding a new vertex within every face, adjacent to all the vertices of the face. Geometrically, it can be formed by attaching a pyramid to every face, simultaneously. An \emph{iterated Kleetope} is the result of repeatedly applying this operation a given number of times.
Following notation from our previous work~\cite{Epp-GCOM-21}, we
let $KG$ denote the Kleetope of a graph $G$ and $K^iG$ denote the result of applying the Kleetope operation $i$ times to~$G$.

\paragraph{Thickness} is the minimum number of planar subgraphs needed to cover all edges of a given graph. Equivalently, it is the minimum number of edge colors needed to draw the graph in the plane with colored edges so each crossing has edges of two different colors. A graph is \emph{biplanar} if its thickness is at most two. Thus, a biplanar drawing of a graph can be interpreted as a pair of planar drawings of two subgraphs of the given graph that, together, include all of the graph edges.
Repeated edges are never necessary and for technical reasons we forbid them.

Little was known about the structure of biplanar graphs. They include all graphs of maximum degree four~\cite{Hal-IS-91,DunEppKob-SoCG-04}, and therefore cannot have more structure than degree-four graphs. An NP-completeness reduction for biplanarity by Mansfield~\cite{Man-MPCPS-83} can be used to construct infinitely many non-biplanar graphs. Additionally, as S{\'y}kora et al. observed, 5-regular graphs of girth $\ge 10$ are too dense for their girth to be biplanar~\cite{SykSzeVrt-IS-04}.
 
\emph{Split thickness} is a generalization of thickness in which we form a single planar drawing with multiple copies of each vertex, which are not required to be near each other in the drawing. Each edge of the graph appears once, connecting an arbitrary pair of copies of its endpoints. A drawing has split thickness $k$ if each vertex has at most $k$ copies, and the split thickness of a graph $G$ is the minimum number $k$ such that $G$ has a drawing with split thickness~$k$~\cite{EppKinKob-Algo-18}. Split thickness is less than or equal to thickness, but they can diverge, even for complete graphs: $K_{12}$ has split thickness two~\cite{Hea-QJM-90} but $K_9$ already has thickness three~\cite{BatHarKod-BAMS-62,Tut-CMB-63}. Thickness has its origin in the Earth--Moon problem, posed by Gerhard Ringel in 1959~\cite{Rin-59}, which in graph-theoretic terms asks for the maximum chromatic number of biplanar graphs. In the same way, split thickness corresponds to the older $m$-pire coloring problem~\cite{Gar-SA-80}.

\paragraph{Blowups} of graphs are formed by duplicating their vertices a given number of times. More specifically, the (open) $k$-blowup of a graph $G$, which we denote as $\blowup{k}{G}$,\footnote{Blowups are a standard concept but their notation varies significantly. Other choices from the literature include $G(k)$, $G[k]$, $G^k$, and $G^{(k)}$.} is obtained by making $k$ copies of each vertex of $G$, and by connecting two vertices in $\blowup{k}{G}$ whenever they are copies of adjacent vertices in $G$. Two copies of the same vertex are not adjacent. In a \emph{closed blowup}, the copies are adjacent.

\begin{figure}[t]
\centering\includegraphics[width=\textwidth]{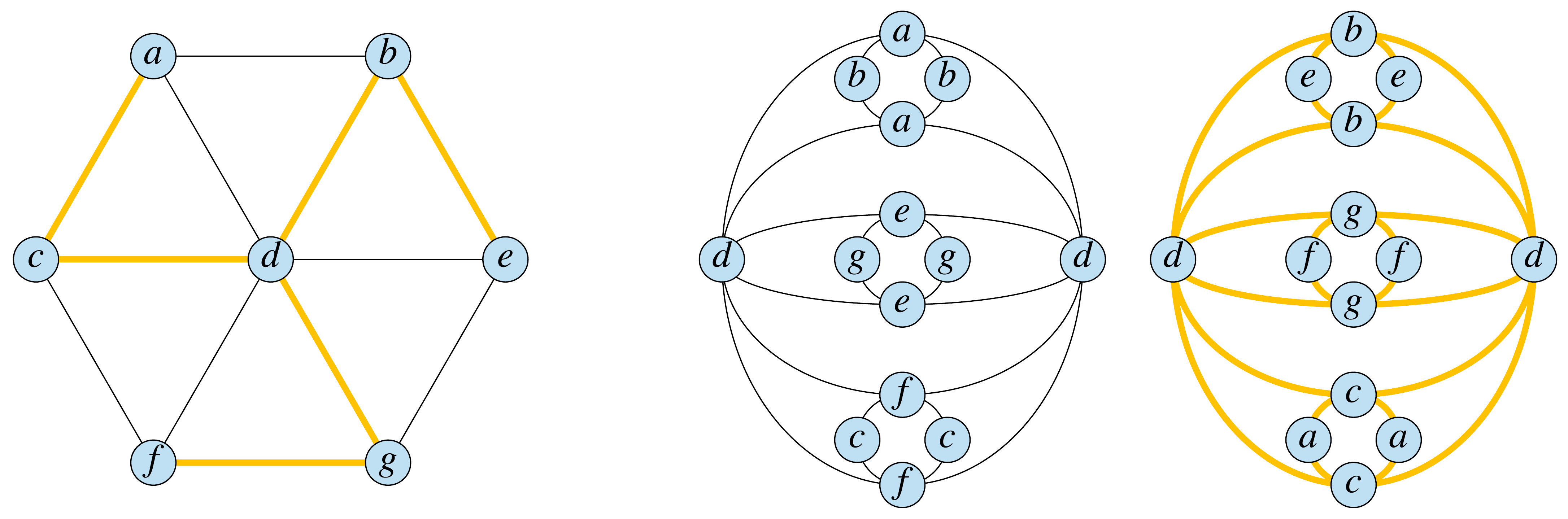}
\caption{Decomposition of a seven-vertex wheel graph into two trees (left) and the corresponding biplanar drawing of its 2-blowup (right)}
\label{fig:thick-arbor}
\end{figure}

Albertson, Boutin, and Gethner~\cite{AlbBouGet-DM-10} proved that the closed 2-blowup of any tree or forest is planar and therefore that the closed 2-blowup of any graph of arboricity $a$ has thickness at most $a$. Here, \emph{arboricity} is the minimum number of forests that cover all edges of a graph. Adding a leaf to a forest corresponds, in the closed 2-blowup, to gluing a $K_4$ subgraph onto an edge, which preserves planarity, and the $a$ forests that cover a graph of arboricity $a$ can be blown up by induction in this way, separately from each other (\cref{fig:thick-arbor}).

Planar graphs have arboricity at most three~\cite{Nas-JLMS-61}, and this is tight for planar graphs with more than $2n-2$ edges, including most maximal planar graphs. Therefore, their 2-blowups have thickness at most three. Gethner's conjecture asks whether there are planar graphs for which the resulting thickness-three drawing is optimal or whether smaller thickness, two, can always be achieved.

\subsection{New results}
Our main result is that for all sufficiently large maximal planar graphs $G$, $\blowup{2}{K^3G}$ does not have thickness two and does not have split thickness two. This gives a proof of non-biplanarity for a natural and sparse class of graphs that (unlike previous methods for proving non-biplanarity) allows short cycles.

To complement this result, we provide biplanar or split thickness two drawings for the blowups of three natural classes of planar graphs:
\begin{itemize}
\item We construct a biplanar drawing of the 2-blowup of any planar graph whose faces can be decomposed into two \emph{outerpaths}, strips of polygons connected edge-to-edge with the topology of a path (see \cref{def:outerpath}). Equivalently, this structure is a partition of the dual graph into two induced paths. 
\item When $G$ can be decomposed into two outerpaths, we construct a split thickness two drawing of its Kleetope $KG$. In the special case of the tetrahedral graph $K_4$, this construction can be used for the iterated Kleetope $K^2 K_4$.
\item When $G$ and its dual have disjoint Hamiltonian paths, we construct a split thickness two drawing of the 2-blowup of $G$.
\item We construct a split thickness two drawing of the 2-blowup of any 3-chromatic planar graph, and more generally a drawing with split thickness $k$ of the $k$-blowup of these graphs.
\end{itemize}
 
These drawing algorithms motivate our use of iterated Kleetopes in constructing non-biplanar blowups of planar graphs, because Kleetopes are far from having the properties needed to make these algorithms work. Kleetopes of maximal planar graphs are far from 3-chromatic: each added vertex forms a $K_4$ subgraph, an obstacle to 3-coloring. And iterated Kleetopes are far from being decomposable into outerpaths, as their dual graphs have no long induced paths. The underlying planar graphs for each of our drawing algorithms include infinitely many maximal planar graphs, showing that it is not merely the large size and maximality of iterated Kleetopes that prevents their blowups from having drawings. Additionally, because triangle-free planar graphs are 3-chromatic~\cite{Gro-WZMLU-59}, these constructions suggest that the connection of S{\'y}kora et al.  between girth and non-biplanarity is unlikely to help construct planar graphs with non-biplanar blowups.

\section{Iterated Kleetopes}

In this section we show that some 2-blowups of iterated Kleetopes are not biplanar and do not have split thickness two or less.
As in our previous work on the geometric realization of iterated Kleetopes~\cite{Epp-GCOM-21}, our
approach uses the observation that any realization or drawing of an iterated Kleetope must be based on a realization or drawing of a graph with one fewer iteration. This simpler drawing can be recovered from the final drawing by removing the vertices added in the Kleetope process. Using this observation,
we build up a sequence of stronger properties for the biplanar and split thickness two embeddings of these graphs, as the number of Kleetope iterations increases. Eventually, these properties will become so strong that they lead to an impossibility.

\begin{definition}
For a vertex $v$ of graph $G$, it is convenient to denote the two copies of $v$ in $\blowup{2}{G}$ by $v_0$ and $v_1$. We distinguish these from the two \emph{images} of $v_0$ and the two images of $v_1$ in a biplanar or split thickness two drawing of $\blowup{2}{G}$. In such a drawing, $v$ itself has four images, two from $v_0$ and two from~$v_1$.
\end{definition}

We need the following definitions in the proof of our first lemma.

\begin{definition}
Define the \emph{excess} of a face in a planar, biplanar, or split thickness two drawing of a graph to be the number of edges in the face, minus three, so triangles have excess zero and all other faces have positive excess. Define the total excess of the drawing to be the sum of all face excesses. 
\end{definition}

The total excess of a drawing equals the amount by which the number of edges in the graph falls short of the maximum possible number of edges in a drawing of its type, and so can be calculated only from a graph and the type of its drawing, independent of how it is drawn:
\begin{itemize}
\item A planar drawing of an $n$-vertex graph can have at most $3n-6$ edges, and if there are $m$ edges then the total excess is $(3n-6)-m$.
\item A biplanar drawing of an $n$-vertex graph can have at most $6n-12$ edges, and if there are $m$ edges then the total excess is $(6n-12)-m$.
\item A split thickness two drawing of an $n$-vertex graph can have at most $6n-6$ edges, and if there are $m$ edges then the total excess is $(6n-6)-m$.
\end{itemize}

\begin{lemma}
\label{lem:triangulated}
Let $G$ be a maximal planar graph with $n$ vertices. If $n\ge 49$, then in any biplanar drawing $D$ of $\blowup{2}{G}$ some vertex $v$ of $G$ has images that are only incident to triangles. If $n\ge 73$, then for any split thickness two drawing some vertex $v$ has the same property. We say that $v$ has \emph{triangulated neighborhoods}.
\end{lemma}

\begin{proof}
The number of vertices that belong to non-triangular faces in a planar, biplanar, or split thickness two drawing is maximized when the the total excess is distributed among disjoint quadrilateral faces. In this case, the number of such vertices is four times the excess. Any other distribution of the total excess, or non-disjointness among the non-triangular faces, produces fewer such vertices.

Because $G$ is maximal planar, it has excess zero and $3n-6$ edges. Its blowup $\blowup{2}{G}$ has $2n$ vertices and $12n-24$ edges, four copies of each edge in $G$. Therefore, any biplanar drawing of $G$ has excess $12$, and any split thickness two drawing of $G$ has excess $18$. The number of vertices that can belong to a non-triangular face in these drawings is, respectively, $48$ and $72$. For larger values of $n$, some vertex has triangulated neighborhoods.
\end{proof}

\begin{figure}[t]
\centering\includegraphics[scale=0.25]{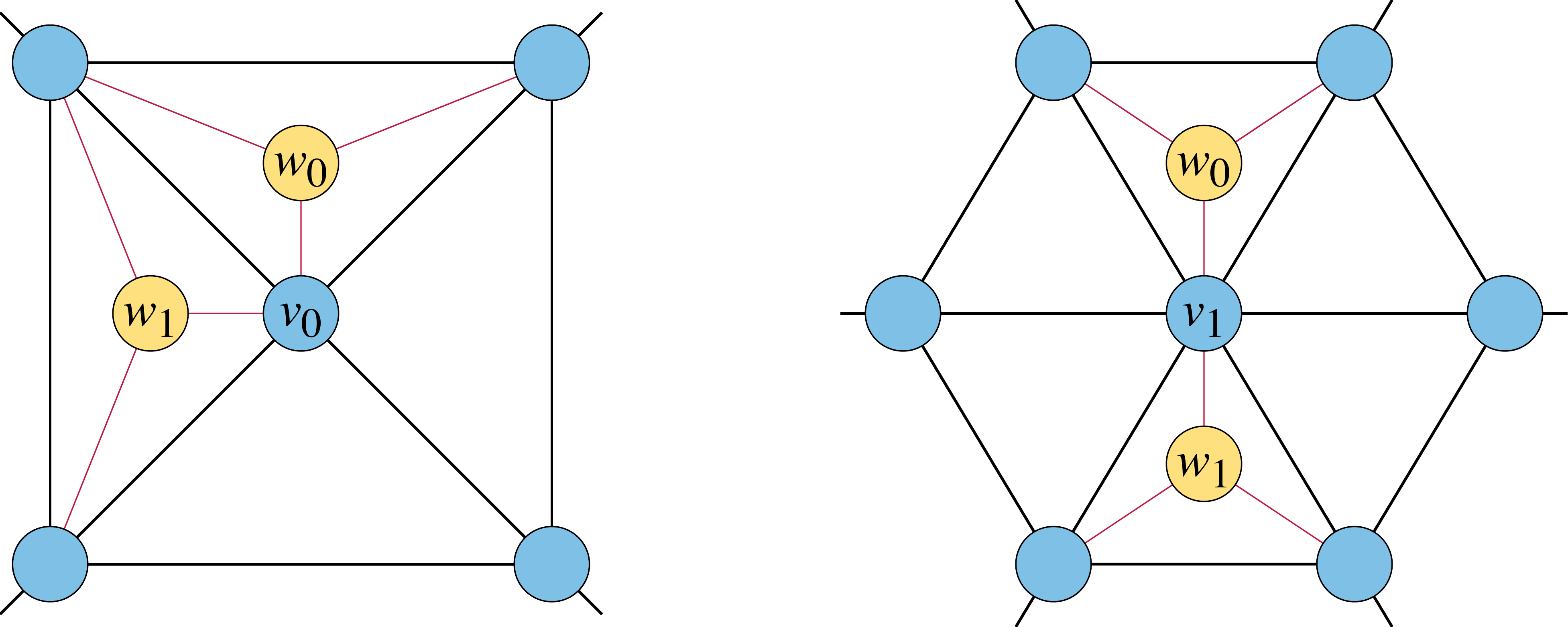}
\caption{Left: Illustration for \cref{lem:triangular}: If $v$ in $G$ has triangulated neighborhoods, and $w$ is any neighbor of $v$ added in $KG$, then the images of $w$ must lie in two triangles incident to images of $v$, connected to all six triangle vertices. In this example, the four images of $w$ are neighbors of only two images of $v$, but they may instead be neighbors of three or four images of $v$.}
\label{fig:triangulate}
\end{figure}

\begin{lemma}
\label{lem:triangular}
Let $G$ be a maximal planar graph, and consider any biplanar drawing or split thickness two drawing $D$ of $\blowup{2}{KG}$,
and the restriction of the same drawing to $G$.
If some vertex $v$ of $G$ has triangulated neighborhoods in the restriction to $G$, then any neighbor $w$ of $v$ in $KG\setminus G$ has its four images each drawn surrounded by exactly three triangular faces. We say that $w$ has \emph{triangular neighborhoods}.
If $G$ has $n$ vertices with $n\ge 49$, then in any biplanar drawing $D$ of $\blowup{2}{KG}$ some vertex $w$ of $KG$ has triangular neighborhoods. If $n\ge 73$, then for any split thickness two drawing of $\blowup{2}{KG}$ some vertex $w$ has triangular neighborhoods.
\end{lemma}

\begin{proof}
As an added vertex in $KG$, $w$ has degree three, so its copy $w_0$ in $\blowup{2}{KG}$ has degree six. To be adjacent to both $v_0$ and $v_1$, the two images of $w_0$ must lie in two triangular faces of the restricted drawing containing $v_0$ and $v_1$, as shown in \cref{fig:triangulate}. (No face contains both $v_0$ and $v_1$, because they have triangular neighborhoods and are not adjacent.)  This placement limits $w_0$ to having as neighbors only the six vertices of these two triangles, matching its degree, so it must be connected to all six of these vertices. The same argument applies to $w_1$.

The existence of $w$ in biplanar or split thickness drawings of $\blowup{2}{KG}$ for graphs with many vertices follows by applying this argument to the vertex $v$ with triangulated neighborhoods given by \cref{lem:triangulated}.
\end{proof}

In \cref{lem:triangular}, the two triangular neighborhoods of $w_0$ must be disjoint, so they cover all six distinct neighbors of $w_0$ in $\blowup{2}{KG}$. Similarly, the two neighborhoods of $w_1$ must be disjoint. However, a neighborhood of $w_0$ may share a vertex or an edge with a neighborhood of $w_1$. (It cannot share edges with two neighborhoods because then those two neighborhoods would not be disjoint.) In fact, some sharing is necessary:

\begin{lemma}
\label{lem:no-edge-disjoint}
Let $t$ be a vertex of a planar graph $H$ having three neighbors, all adjacent. Suppose that $t$ has triangular neighborhoods in a biplanar or split thickness two drawing of $\blowup{2}{H}$. Then it is impossible for all four images of $t$ to have edge-disjoint neighborhoods in this drawing.
\end{lemma}

\begin{proof}
Let $\Delta$ be the triangle of neighbors of $t$ in $H$.  $\blowup{2}{\Delta}$ is isomorphic to $K_{2,2,2}$, the graph of a regular octahedron. We are assuming our drawings have no repeated edges, so two images of $t$ with edge-disjoint neighborhoods in the drawing must come from edge-disjoint triangles of  $\blowup{2}{\Delta}$. Thus, if the four images of $t$ could be drawn with edge-disjoint triangular neighborhoods, these neighborhoods would form four edge-disjoint triangles of $\blowup{2}{\Delta}$. But in any subdivision of $\blowup{2}{\Delta}$ into four edge-disjoint triangles (\cref{fig:tricover-unshare}, left), each two triangles share a vertex, and together cover only five vertices of $\blowup{2}{\Delta}$. If the four images of $t$ were placed in images of these four triangles, the two triangular neighborhoods of $t_0$ would miss one of the six neighbors of $t_0$ in $\blowup{2}{H}$, as would the two triangular neighborhoods of $t_1$, preventing the drawing from being valid. Therefore, no such drawing is possible.
 \end{proof}

\begin{figure}[t]
\centering\includegraphics[scale=0.25]{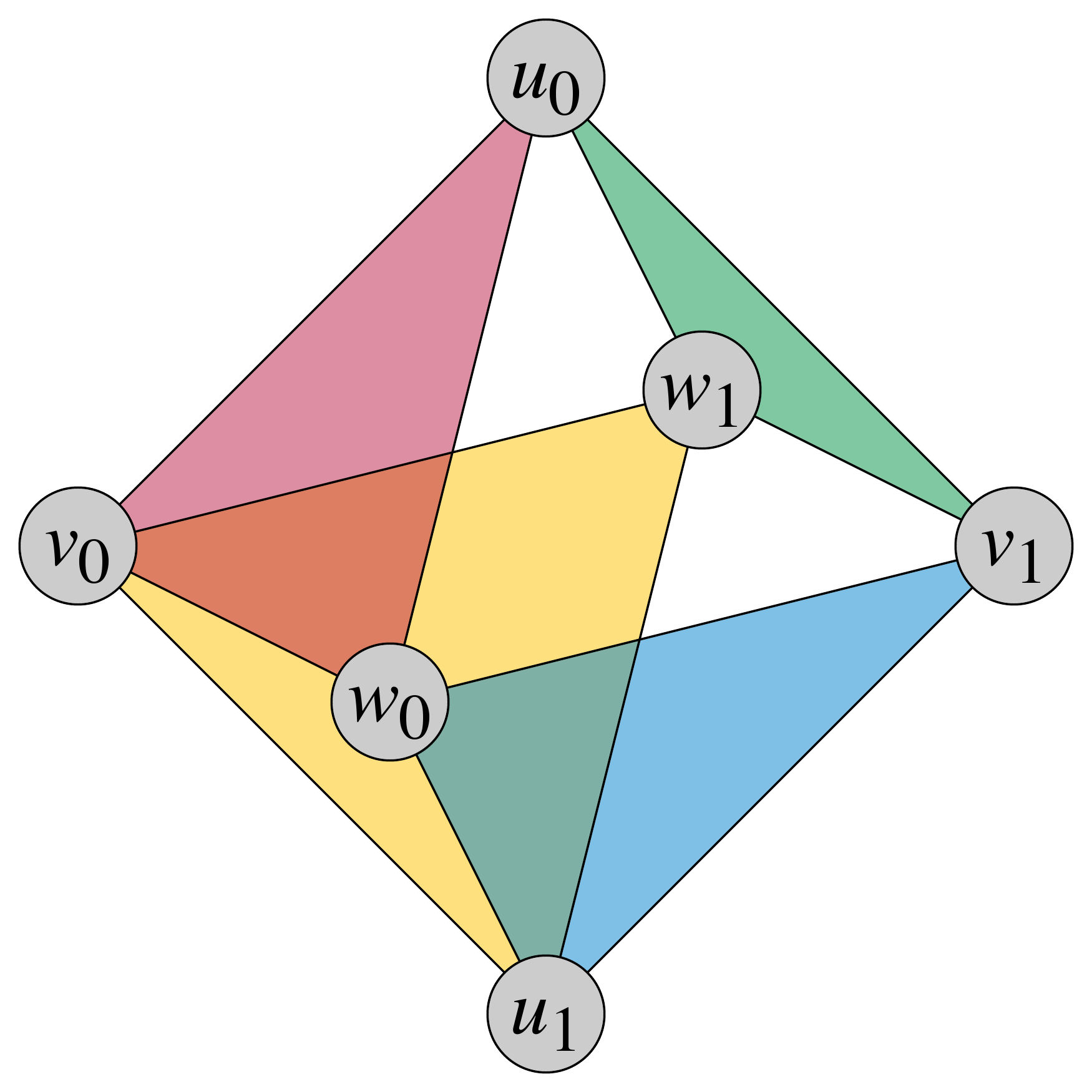}
\qquad\includegraphics[scale=0.25]{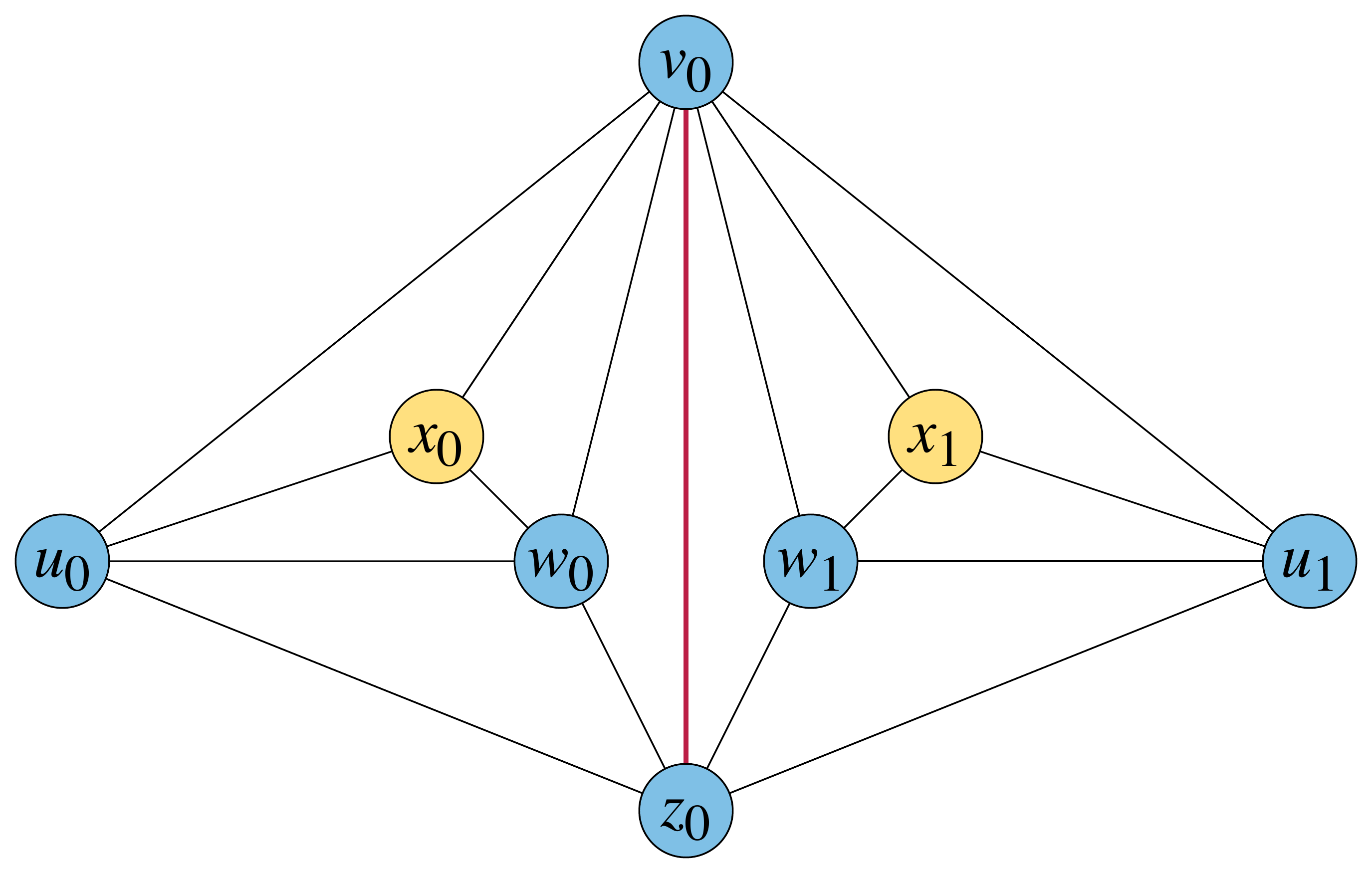}
\caption{Left: Illustration for \cref{lem:no-edge-disjoint}: partition of $\blowup{2}{\Delta}$ into four triangles, for a triangle $\Delta=uvw$. Right: Illustration for \cref{thm:kleetope}: For the restriction of a given drawing to $\blowup{2}{KG}$, and for $w$ in $KG$, images $w_0$ and $w_1$ have triangular neighborhoods sharing edge $v_0z_0$ (red). The third vertices of these triangular neighborhoods, $u_0$ and $u_1$, are distinct images of a neighbor $u$ of $w$. For vertex $x$ in $K^2G$ adjacent to $w$ and to $u$, two images have triangular neighborhoods without shared edges.}
\label{fig:tricover-unshare}
\end{figure}

\begin{theorem}
\label{thm:kleetope}
Let $G$ be a maximal planar graph with $n$ vertices. If $n\ge 49$, then $\blowup{2}{K^3G}$ has no biplanar drawing, and if $n\ge 73$, then $\blowup{2}{K^3G}$ has no split thickness two drawing.
\end{theorem}

\begin{proof}
Suppose for a contradiction that such a drawing existed, and consider the drawings within it of $\blowup{2}{K^2G}$ and of $\blowup{2}{KG}$. We will find a vertex with triangular neighborhoods in each of these drawings so that the four neighborhoods of the four images of the chosen vertex are nearly disjoint: these neighborhoods can share at most two edges total in $\blowup{2}{KG}$, at most one edge in $\blowup{2}{K^2G}$, and no edges in $\blowup{2}{K^3G}$. The existence of a vertex in $K^3G$  whose triangular neighborhoods share no edges will contradict \cref{lem:no-edge-disjoint}, showing that no such drawing can exist. To do this, we consider each level of iteration successively, as follows:
\begin{itemize}
\item In the drawing of $\blowup{2}{KG}$, \cref{lem:triangular} gives us a vertex $w$ of $KG$ with triangular neighborhoods. The neighborhood of each image of $w$ shares at most one edge with other neighborhoods of images of $w$. For biplanar drawings this is immediate (only one other image is in the same planar subgraph and can share an edge with it). For split thickness two drawings, a neighborhood of $w_0$ that shares edges with the neighborhoods of both images of $w_1$ is impossible, because then the two images of $w_1$ would share a vertex, preventing them from covering all six neighbors of $w_1$. Therefore, among the neighborhoods of all four images of $w$, there are at most two shared edges.

\smallskip
\item To find a vertex $x$ in $K^2G$ with triangular neighborhoods in the drawing of $\blowup{2}{K^2G}$, with at most one edge shared among the neighborhoods of its four images, consider the vertex $w$ found above in $KG$. If the neighborhoods of $w$ have at most one shared edge in the drawing of $\blowup{2}{KG}$, let $x$ be any neighbor of $w$ added in forming $K^2G$ from $KG$. Then $x$ must again have triangular neighborhoods by \cref{lem:triangular}. Because these triangular neighborhoods must be interior to the triangular neighborhoods of $w$, they can have at most one shared edge (the same edge as the one shared by the neighborhoods of $w$).

\smallskip
Suppose, on the other hand, that the triangular neighborhoods of $w$ share exactly two edges. Let $\Delta_0$, $\Delta_0'$, $\Delta_1$, and $\Delta_1'$ be the four triangles in $\blowup{2}KG$ neighboring the two images of $w_0$ and $w_1$, respectively, with an edge shared by $\Delta_0$ and $\Delta_1$ and another edge shared by $\Delta_0'$ and $\Delta_1'$. Because these triangles can only share one edge, and each image of $w$ must be adjacent to images of all three neighbors of $w$, the vertices of $\Delta_0$ and $\Delta_1$ that are not on the shared edge must be distinct images of the same vertex $u$. Choose a vertex $x$ of $K^2G\setminus KG$, adjacent to $w$ and to $u$.

\smallskip
Because $w$ has triangular neighborhoods in $KG$, its four images in the drawing of $\blowup{2}KG$ are each surrounded by three triangles,
formed by images of $w$ and two of its neighbors. For each image, only one of these triangles consists of the three neighbors of $x$. Thus, the four images of $x$ must be placed in these four triangles. Two of these four triangles are subdivisions of $\Delta_0$ and $\Delta_1$, containing the non-shared images of $u$, and therefore do not share any edge with each other. The only possible shared edge among the triangular neighborhoods of $x$ is the edge shared by $\Delta_0'$, and $\Delta_1'$, within which lie the other two images of $x$. Thus, by choosing $x$ in $K^2G$ we have eliminated one shared edge between triangular neighborhoods. See \cref{fig:tricover-unshare}, right.

\smallskip
\item If the four triangular neighborhoods of $x$ in the drawing of $\blowup{2}K^2G$ are edge-disjoint, we already have a contradiction with \cref{lem:no-edge-disjoint}. Otherwise, we must find a vertex $t$ in $K^3G$ whose triangular neighborhoods are edge-disjoint, giving us the desired contradiction. To do so, consider the four triangular neighborhoods of $x$ in the drawing of $\blowup{2}K^2G$, only two of which share one edge.
For the two triangles that share an edge, the two non-shared vertices of these triangles must be distinct images of the same vertex $y$ in $K^2 G$.
Choose a vertex $t$ of $K^3G\setminus K^2G$, adjacent to $x$ and to $y$. Then the four triangles in which $z$ must be placed lie within the four triangular neighborhoods of $x$, away from the shared edge of these triangular neighborhoods, so they cannot share any edges with each other. By \cref{lem:no-edge-disjoint}, this is an impossibility.
\end{itemize}
\end{proof}

When $G$ is not maximal planar or is too small for the theorem to apply directly, more iterations of the Kleetope operation can be used before applying the same argument. Thus, there exists an $i$ such that for every plane graph $G$, $\blowup{2}{K^iG}$ is not biplanar and has no split thickness two drawing.

\section{Drawings from Triangle Strips}

\begin{definition}
\label{def:outerpath}
An \emph{outerpath} is an outerplanar graph whose weak dual (the adjacency graph of its bounded faces) is a path.
\end{definition}

Suppose that the dual vertices of a planar graph $G$ can be partitioned into two subsets that each induce a path in the dual graph. Then the dual cut edges between these two subsets correspond, in $G$ itself, to a Hamiltonian cycle that partitions $G$ into two outerpaths. \cref{fig:2strip-icos} depicts an example of this sort of \emph{two-outerpath decomposition} for the graph of the regular icosahedron.

\begin{theorem}
\label{thm:outerpaths}
If a planar graph $G$ has a two-outerpath decomposition, then $\blowup{2}{G}$ has a biplanar drawing.
\end{theorem}

\begin{figure}[t]
\centering\includegraphics[scale=0.25]{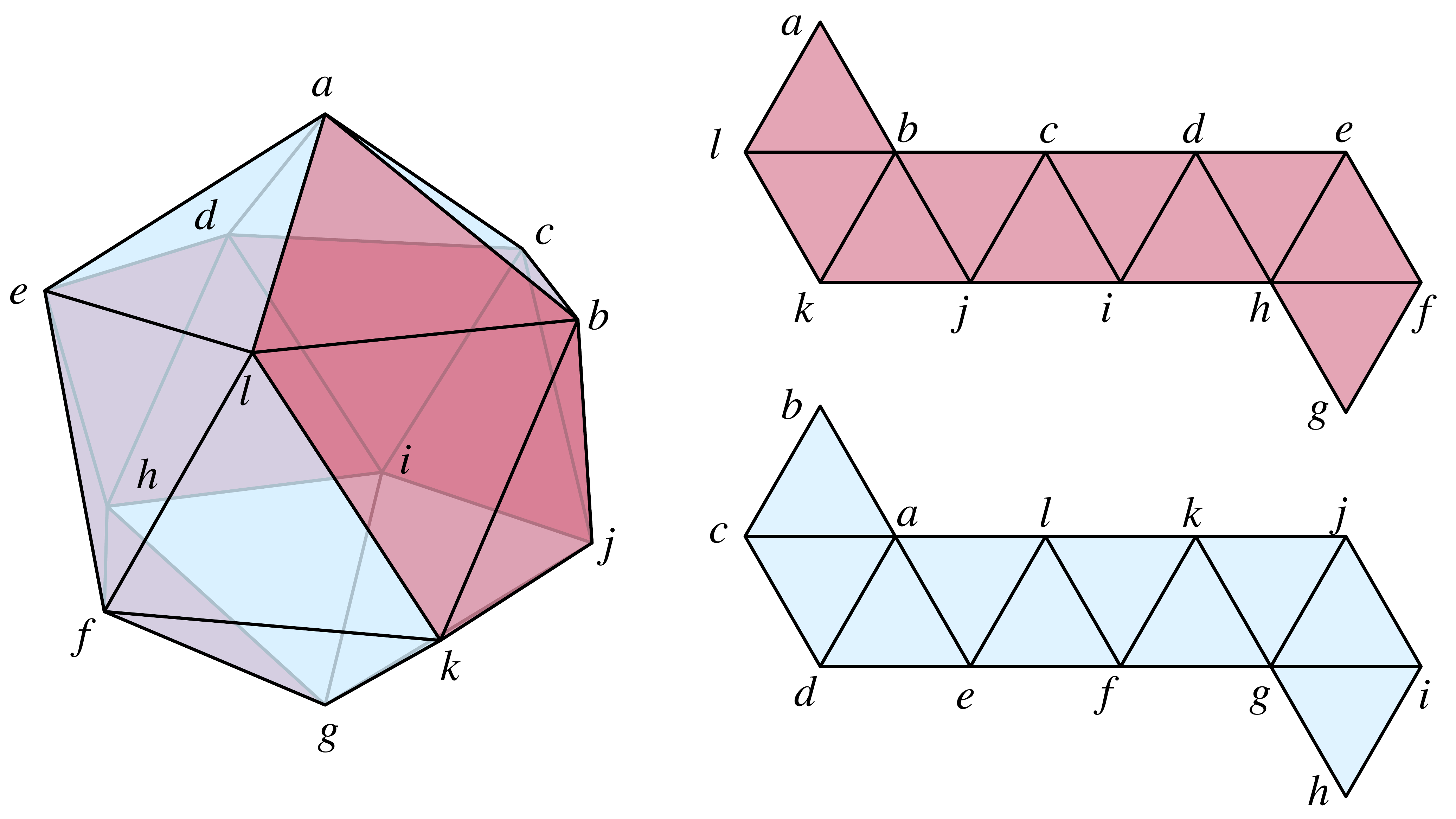}
\caption{Decomposition of an icosahedron into two outerpaths.}
\label{fig:2strip-icos}
\end{figure}

\begin{figure}[t]
\centering\includegraphics[scale=0.35]{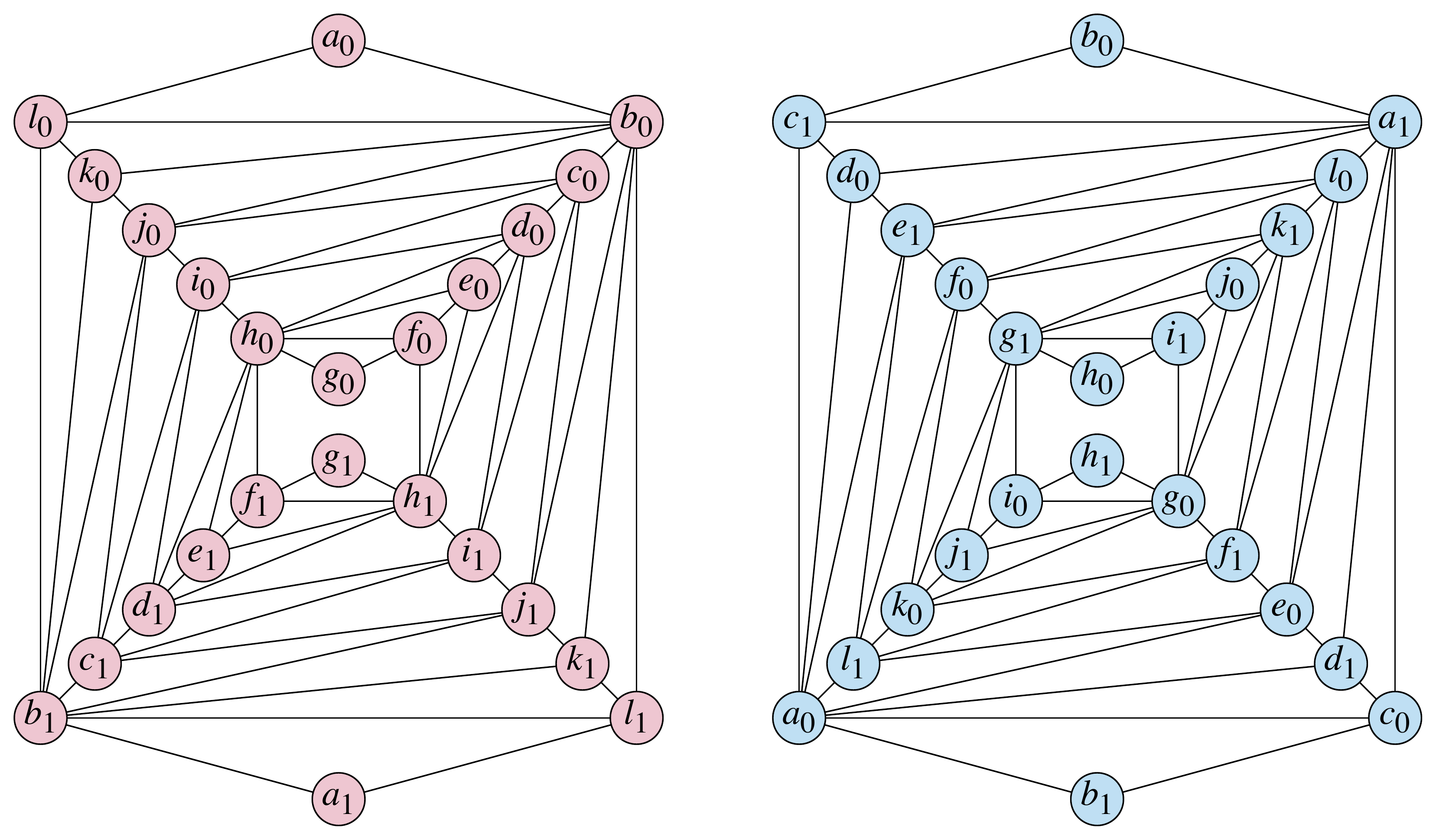}
\caption{Biplanar drawing of the 2-blowup of an icosahedron corresponding to the outerpath decomposition of \cref{fig:2strip-icos}.}
\label{fig:icos-blowup}
\end{figure}

\begin{proof}
We may assume without loss of generality, by triangulating each outerplanar graph if necessary, that both outerpaths are maximal outerplanar: each of their faces is a triangle. If this triangulation step adds two copies of the same edge to the graph, it is not a problem, because the edges added in this triangulation step will be removed from the final drawing.

In each outerpath, the triangular faces form a linear sequence, separated by the internal edges of the outerpath, which are also linearly ordered. In the blowup~$\blowup{2}{G}$, number the two copies of each vertex $v$ as $v_0$ and $v_1$. If $uv$ is one of the diagonals of one of the outerpaths (that is, one of its internal edges), then~$\blowup{2}{uv}$ is a four-vertex cycle $u_0v_0u_1v_1$.

We will construct a biplanar drawing of $\blowup{2}{G}$ with each plane containing all copies of interior edges of one of the two outerpaths, and two out of the four copies of each boundary edge. We draw copies of the interior edges as nested quadrilaterals, one for each diagonal of its outerpath, in the same order that these diagonals appear within the outerpath. If we draw these quadrilaterals one at a time, from the innermost to the outermost, then each two consecutive quadrilaterals share two opposite vertices, corresponding to the single shared endpoint of the two diagonals. 

In each pair of consecutive quadrilaterals, the outer quadrilateral has two potential orientations with respect to the inner one: if the two consecutive diagonals are $uv$ and $vw$, with quadrilateral $u_0v_0u_1v_1$ drawn inside quadrilateral $v_0w_0v_1w_1$, then these quadrilaterals may be drawn so that pairs $u_0w_0$ and $u_1w_1$ are adjacent, or so that pairs $u_0w_1$ and $u_1w_0$ are adjacent. In one plane we always choose the orientation with $u_0w_0$ and $u_1w_1$ adjacent, and we connect these pairs of vertices by an edge. In the other plane, we always choose the orientation with $u_0w_1$ and $u_1w_0$ adjacent, and we connect these pairs of vertices by an edge. In this way, we draw all four copies of each boundary edge, except the edges incident to the two ears (triangles with two boundary edges).

It remains to draw the ears. Each has two boundary edges sharing a vertex, which we call the \emph{ear vertex}. These two boundary edges have not yet been drawn in the plane of their outerpath.  The third edge of the ear is a diagonal whose images form the innermost or outermost quadrilateral in its plane. We place both copies of the ear vertex inside or outside this quadrilateral (respectively as it is innermost or outermost), connected to its two neighbors in the ear with the same numbering convention:
in the plane where the quadrilaterals are oriented with $u_0w_0$ and $u_1w_1$ adjacent, we connect each ear vertex to neighbors with the same subscript, and in the other plane we connect each ear vertex to neighbors with the opposite subscript.

Thus, all copies of the diagonals of one strip and all copies of boundary edges that connect copies having the same index are drawn in one plane. All copies of the diagonals of the other strip and all copies of boundary edges that connect copies having different indices are drawn in the other plane. The result is a biplanar drawing of the entire blowup $\blowup{2}{G}$.
\end{proof}

\cref{fig:icos-blowup} depicts the drawing obtained by applying \cref{thm:outerpaths} to the outerpath decomposition of the icosahedron depicted in \cref{fig:2strip-icos}. The following extension of this result is noteworthy in connection with our iterated Kleetope counterexample:

\begin{theorem}
\label{thm:split-kleetope}
If a maximal planar graph $G$ has a decomposition into two outerplanar graphs, coming from an induced path partition of its dual graph into two paths, then $\blowup{2}{KG}$ has a drawing with split thickness two.
\end{theorem}

\begin{proof}
From the drawing of \cref{thm:outerpaths} for $\blowup{2}{G}$, add two more edges from the ear vertices to the other vertices on the hexagonal faces they belong to, so that each of the two planes of the drawing contains four images of each triangular face of its outerplanar subgraph of $G$, in two disjoint pairs. These four triangles can be used to place the four images of each added vertex of $KG$.
\end{proof}

\begin{figure}[t]
\centering\includegraphics[scale=0.25]{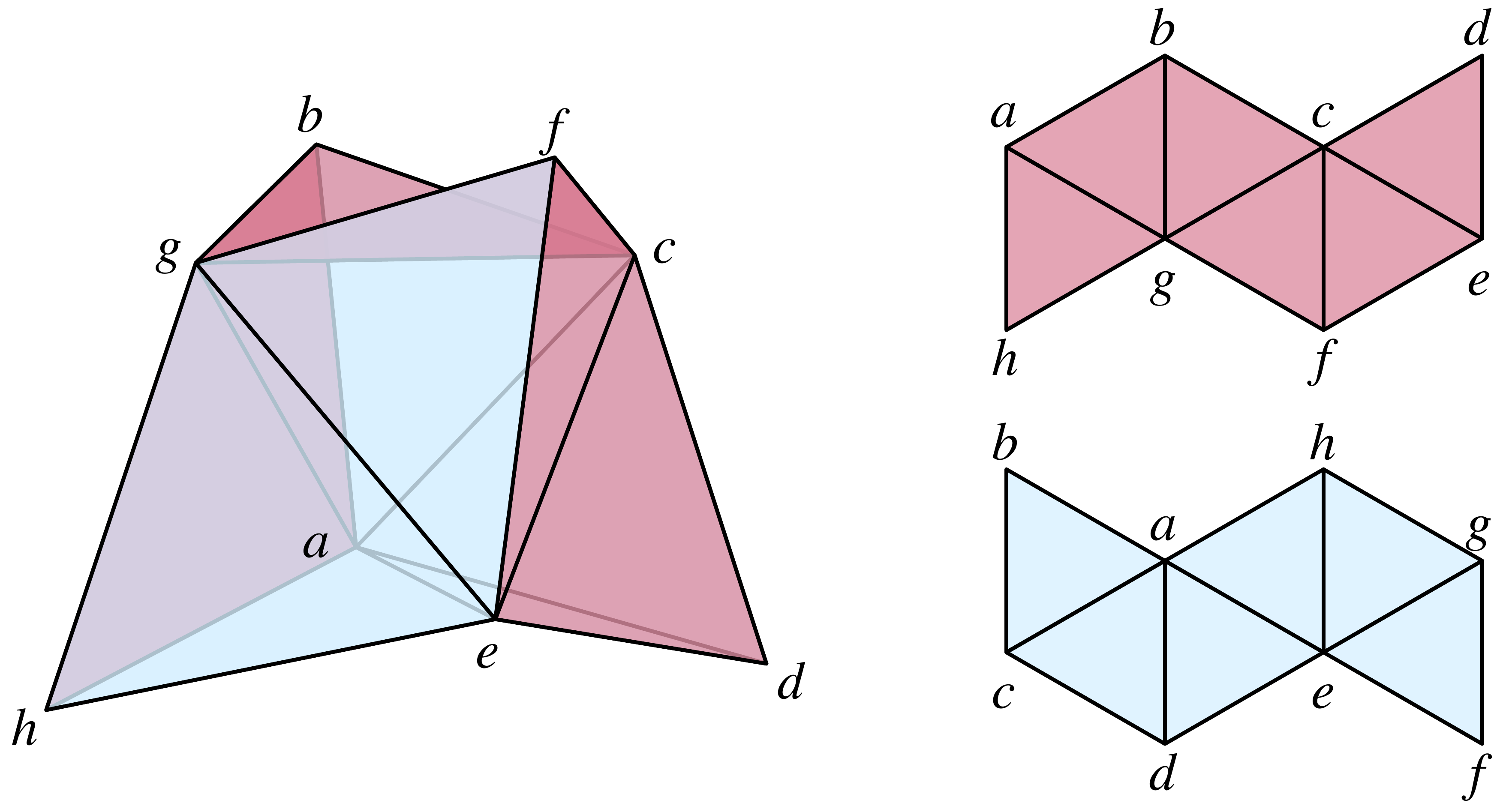}
\caption{Outerpath decomposition of $KK_4$.}
\label{fig:2strip-kk4}
\end{figure}

For instance, the triakis icosahedron, the Kleetope of the icosahedron, is a maximal planar graph whose edges all have total degree $\ge 13$. (This is the maximum possible for the minimum total degree of an edge, by Kotzig's theorem~\cite{Kot-MFC-55}.) Because the icosahedron has a two-outerpath decomposition, we can apply \cref{thm:split-kleetope} to its Kleetope, producing a drawing with split thickness two of the 2-blowup of the triakis icosahedron.

In general, we cannot extend this construction to higher-order Kleetopes. When a graph $G$ has a two-outerpath decomposition, the drawings of $\blowup{2}{KG}$ produced from $G$ by \cref{thm:split-kleetope} again have four images of each triangular face of~$KG$, but some of these quadruples of images cannot be grouped into disjoint pairs. However, in the method of \cref{thm:split-kleetope} each added vertex of $K^2G$ corresponds to two vertices in $\blowup{2}{K^2G}$, and the two images of each of these two vertices must be placed in disjoint triangles, in order to provide all six of its adjacencies. Therefore, this method does not provide drawings of $\blowup{2}{K^2G}$. However, in one special case, for $G=K_4$ (the graph of a tetrahedron), a different method works. In this case, the Kleetope $KK_4$ has an outerpath decomposition, shown in \cref{fig:2strip-kk4}. Therefore, applying \cref{thm:split-kleetope} we can obtain a split thickness two drawing of $\blowup{2}{K^2K_4}$.

A very similar drawing algorithm to the one in \cref{thm:outerpaths} can be used for graphs with a different form of decomposition into triangle strips.

\begin{definition}
Let $G$ be a planar graph having both a Hamiltonian path $P$ and a dual Hamiltonian path $P^*$, with no edge and its dual edge belonging to both paths. Then we call $(P,P^*)$ a \emph{path--copath decomposition}. It is a special case of the \emph{tree--cotree decomposition} formed from any spanning tree of a planar graph and the dual spanning tree formed by the duals of the complementary set of edges~\cite{Epp-SODA-03}.
\end{definition}

\begin{theorem}
\label{thm:path-copath}
Let planar graph $G$ have a path--copath decomposition. Then $\blowup{2}{G}$ has a drawing with split thickness two.
\end{theorem}

\begin{proof}
We may triangulate the faces of $G$, if necessary, preserving the existence of a path--copath decomposition by choosing added diagonals that split each face into an outerpath. As a result, the dual path $P^*$ of the path--copath decomposition $(P,P^*)$ becomes a triangulated outerpath. Each vertex of $G$ may appear multiple times on the boundary of this outerpath, with multiplicity equal to its degree in $P$ (at most two, because $P$ is a path). The outerpath can be formed from $G$ by cutting the plane along each edge of $P$; as a result, each edge of $P$ appears exactly twice on the boundary of the outerpath.

We apply the method of \cref{thm:outerpaths} to this single outerpath, producing a drawing in which each appearance of a vertex $v$ of $G$ on the boundary of the outerpath produces images of both copies of $v$. If $v$ appears once on the boundary of the outerpath, its two copies appear once; if $v$ appears twice, its two copies appear twice, giving this drawing split thickness two. This drawing style automatically produces all four images of each edge interior to the outerpath. For an edge $uv$ of path $P$, appearing twice on the boundary of the outerpath, we choose arbitrarily which appearance of $uv$ on the boundary of the outerpath is used to draw edges $u_0v_0$ and $u_1v_1$, and which is used to draw edges $u_0v_1$ and $u_1v_0$. In this way, all four images of $uv$ are drawn correctly.
\end{proof}

\begin{figure}[t]
\centering\includegraphics[width=\textwidth]{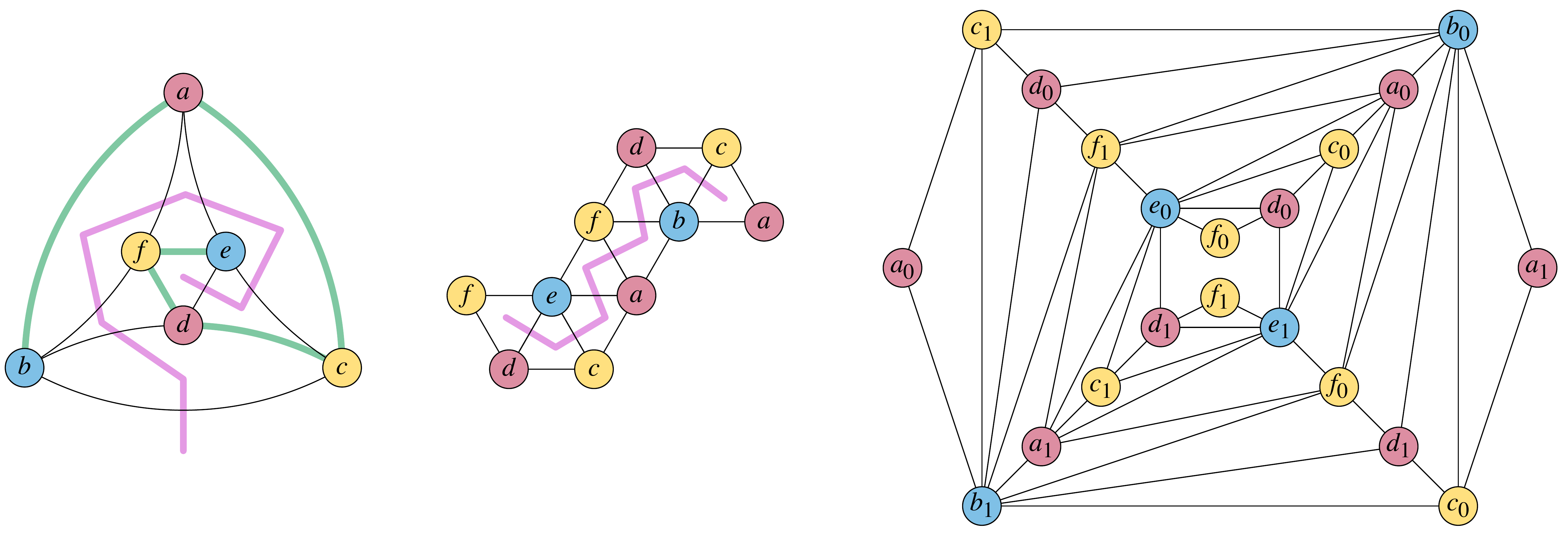}
\caption{Path--copath decomposition of the octahedral graph $K_{2,2,2}$ (left), the outerpath obtained by cutting the path (center), and the split thickness 2 drawing of $\blowup{2}{K_{2,2,2}}=K_{4,4,4}$ obtained from \cref{thm:path-copath} (right).}
\label{fig:path-copath}
\end{figure}

\cref{fig:path-copath} depicts an example.

\section{Drawings from Colorings}

We show in this section that the blowup $\blowup{2}{G}$ of a 3-colored planar graph $G$ has split thickness at most two. We do not know whether all such blowups are biplanar; our construction does not produce a biplanar drawing. More generally, we show that $\blowup{k}{G}$ has split thickness at most $k$; the result for $\blowup{2}{G}$ is a special case.

\renewcommand{\floatpagefraction}{.8}
\begin{figure}[t!]
\centering\includegraphics[width=0.9\textwidth]{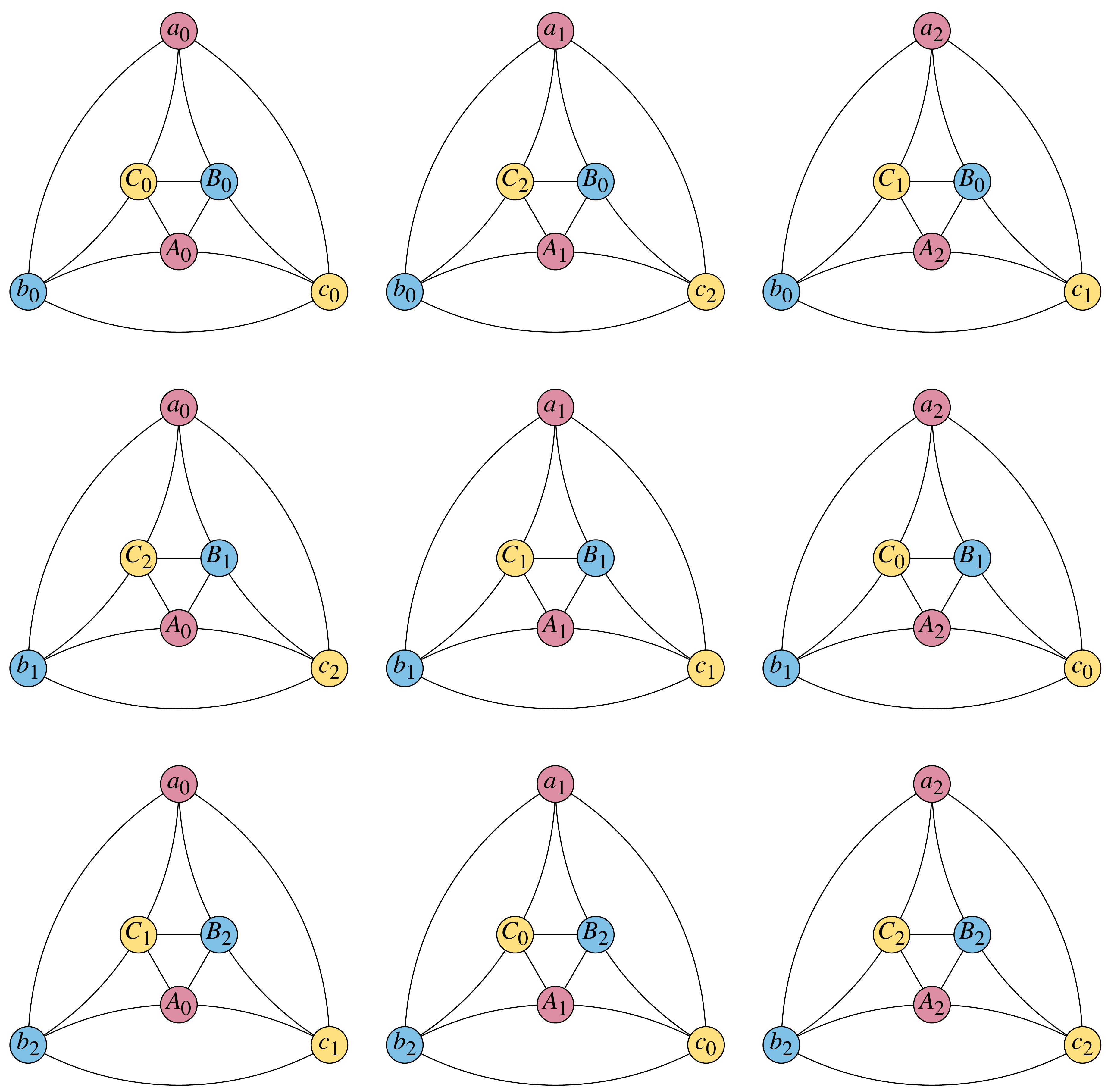}
\caption{\cref{thm:3color} applied to the graph $\blowup{3}{K_{2,2,2}}=K_{6,6,6}.$ Each vertex is labeled with a letter (its position in $K_{2,2,2}$), a number (its index as a copy in the blowup), and a color in the coloring of $K_{2,2,2}$. Each letter-number combination has three images, so this is a drawing with split thickness three. $K_{6,6,6}$ has 108 edges, but drawings of 18-vertex graphs with split thickness two can have at most 102 edges, so this drawing is optimal.}
\label{fig:triple-octahedron}
\end{figure}

\begin{theorem}
\label{thm:3color}
Let $G$ be planar and 3-chromatic; then $\blowup{k}{G}$ has a drawing with split thickness $k$.
\end{theorem}

\begin{proof}
Color the vertices of $G$ red, blue, and yellow, and number the copies of each vertex in $\blowup{k}{G}$ from $0$ to $k-1$.
Draw $\blowup{k}{G}$ as $k^2$ disjoint copies of $G$, where for $(i,j)$ with $0\le i,j<k$ we draw a copy of $G$ consisting of the copies of red vertices numbered $i$, the copies of blue vertices numbered $j$, and the copies of yellow vertices numbered $-(i+j)$ mod $k$. As the disjoint union of $k^2$ planar drawings, the result is planar. Each edge of $\blowup{k}{G}$ appears in one copy of $G$, and each vertex in $\blowup{k}{G}$ has images in $k$ copies of $G$. As a planar drawing with $k$ images of each vertex, it is a drawing with split thickness $k$.
\end{proof}

\cref{fig:triple-octahedron} shows a drawing of $K_{6,6,6}$ with split thickness three, obtained by applying this construction to the triple blowup of the graph of the octahedron. We remark that, when applied to planar bipartite graphs, the same construction yields a thickness $k$ drawing of the $k$-blowup: if we group together the copies of $G$ that would have index $i$ for the vertices of the missing color, then each copy of each vertex appears once in each group. In the case $k=2$, it is also possible to find biplanar drawings of the 2-blowups of planar bipartite graphs in a different way, using the fact that these graphs have arboricity at most two.

\section{Conclusions}
We have shown that 2-blowups of iterated Kleetopes are not biplanar, but that 2-blowups of planar graphs with outerpath decompositions are biplanar. Additionally, we have shown that $2$-blowups of graphs with path--copath decompositions have split thickness at most~2, and $k$-blowups of planar graphs with chromatic number at most three have split thickness at most~$k$.

Several natural questions remain open for future research:
\begin{itemize}
\item Is it ever possible for the 2-blowup of a 3-chromatic planar graph to be non-biplanar? Is it ever possible for the 2-blowup of a 4-vertex-connected planar graph to be non-biplanar?
\item What is the computational complexity of finding biplanar drawings of 2-blowups of planar graphs? Biplanarity is NP-complete in general~\cite{Man-MPCPS-83}, but the proof does not apply to this special case.
\item What is the computational complexity of finding a two-outerpath decomposition? Partition into two induced paths is NP-complete for general graphs~\cite{LeLeMul-DAM-03}, but although its planar case is closely related to Hamiltonicity of the dual graph, we are unaware of complexity results for this case.
\item Can \cref{thm:outerpaths}, on drawing 2-blowups of graphs with a two-outerpath decomposition, be extended from thickness to geometric thickness? Geometric thickness (also called real linear thickness) is similar to thickness, but requires vertices to have the same geometric placement in each planar subgraph and requires edges to be drawn as non-crossing line segments~\cite{DilEppHir-JGAA-00,DujWoo-DCG-07,DunEppKob-SoCG-04,Epp-TTGG-04,Kai-AMSH-73}.
\end{itemize}

\section*{Acknowledgements}

This research was supported in part by NSF grant CCF-2212129.

\bibliographystyle{splncs04}
\bibliography{blowups}

\end{document}